\documentclass[12pt,psamsfonts,leqno,oneside,letterpaper]{amsart}
\usepackage[dvips,text={6.5truein,9truein},left=1truein,top=1truein]{geometry}
\usepackage{amssymb,amsmath,amscd,enumitem}
\usepackage{amscd}
\usepackage[pdftex]{graphicx}
\usepackage{graphicx}
\usepackage{caption}
\usepackage{subcaption}
\usepackage[colorlinks,linkcolor=blue,citecolor=blue,pdfstartview=FitH]{hyperref}
\usepackage{parskip}
\usepackage{color}

\newtheorem{theorem}{Theorem}
\newtheorem{definition}{Definition}
\newtheorem{corollary}{Corollary}
\newtheorem{proposition}{Proposition}

\newcommand\Z{{\mathbb Z}}

\newcommand\F{{\mathbb F}}

\newcommand\R{{\mathbb R}}

\begin{document}
\title{3-Manifolds Modulo Surgery Triangles}
\author{Lucas Culler}
\maketitle

\section{Introduction}

Floer homology theories can be used to define a number of different invariants of closed, oriented $3$-manifolds.  Many of these theories satisfy a ``surgery triangle'', which is a relationship between the invariants of different Dehn surgeries $Y_{\alpha}(K)$ on a knot $K$ in a fixed manifold $Y = Y_{\infty}(K)$.   In general, if $\alpha$, $\beta$, $\gamma$ are three oriented surgery curves on the boundary torus of $Y \setminus \nu(K)$ such that
\[ \alpha \cdot \beta = \beta \cdot \gamma = \gamma \cdot \alpha = - 1 \]
in integral homology, then there is an exact triangle of Floer homology groups,
\[ \cdots \to HF(Y_{\alpha}(K)) \to HF(Y_{\beta}(K)) \to HF(Y_{\gamma}(K)) \to HF(Y_{\alpha}(K)) \to \cdots .\]
Surgery triangles are useful because one can always use them to express the invariants of a given manifold in terms of the invariants of simpler manifolds.  

To make this observation precise, it is useful to borrow some language from finite geometry.  We view diffeomorphism classes of oriented $3$-manifolds as points in a geometry $X$, whose lines are surgery triangles $(Y_{\alpha},Y_{\beta},Y_{\gamma})$.  Note that some lines in this geometry may contain ``doubled'' points, because a manifold can be involved in a surgery triangle with itself (for example, consider Dehn surgeries on the unknot).  We say that a set of $3$-manifolds $S$ is a subspace if each line meeting $S$ in two (not necessarily distinct) points is entirely contained in $S$.  We define the span of a set of $3$-manifolds to be the smallest subspace that contains it.  We say that a set of $3$-manifolds is a generating set if its span is all of $X$.  Using this terminology, we can precisely formulate the sense in which surgery triangles allow one to understand an arbitrary $3$-manifold in terms of simpler manifolds.
\begin{proposition} \label{gen} The $3$-sphere generates all closed $3$-manifolds. \end{proposition}
Similar concepts apply to $3$-manifolds with boundary.
\begin{definition} Let $\Sigma$ be a connected, oriented surface.  A bordered $3$-manifold with boundary $\Sigma$ is a pair $(Y,\phi)$, where $Y$ is a connected, oriented $3$-manifold and $\phi: \partial Y \to \Sigma$ is an orientation preserving diffeomorphism. \end{definition}
One can form a geometry whose points are isomorphism classes of bordered $3$-manifolds with boundary $\Sigma$, and whose lines are surgery triangles.  Denote this geometry by $X(\Sigma)$.  Then Proposition \ref{gen} has the following generalization, due to Baldwin and Bloom \cite{BaBl}. 

\begin{theorem} \cite{BaBl} $X(\Sigma_g)$ is finitely generated.  In fact, it has a generating set $S_g$ of cardinality
\[ N(g) = \sum_{k=0}^g {g \choose k} C_k  \]
where $C_k = \frac{1}{k+1} {2k \choose k}$ is the $k$-th Catalan number. \end{theorem}

In this paper we prove that some elements of $S_g$ can be removed.

\begin{theorem} \label{minimalgen} $X(\Sigma_g)$ has a generating set $M_g \subset S_g$ of cardinality
\[ n(g) = \frac{(2^g +1)(2^{g-1}+1)}{3}. \] 
\end{theorem}

Definitions of $M_g$ and $S_g$ can be found in Section $2$, and a proof of Theorem \ref{minimalgen} can be found in Section $3$.  It turns out that the generating set $S_g$ has two special redundancies, one in genus $5$ and another in genus $6$.  These redundancies imply many more redundancies in higher genera, so the difference between $n(g)$ and $N(g)$ grows rapidly with $g$.  We include for convenience a table of the numbers $n(g)$ and $N(g)$:
\begin{center}
  \begin{tabular}{ | l | c | c | c | c | c | c | c | c |}
    \hline
    g & 0 & 1 & 2 & 3 & 4 & 5 & 6 & 7 \\ \hline
    N(g) & 1 & 2 & 5& 15 & 51 & 188 & 731 & 2950 \\ \hline
    n(g) & 1 & 2 &5 & 15 & 51 & 187 & 715 & 2795 \\
    \hline
  \end{tabular}
\end{center}

It turns out that there are no further redundancies - the generating set $M_g$ is minimal.  In Section $4$ we will prove an even stronger result:

\begin{theorem} \label{minimality} Any generating set for $X(\Sigma_g)$ has cardinality at least $n(g)$.  \end{theorem}

To explain the concept behind our proof, it is useful to think about a sort of classifying space.  Define a CW complex $B(\Sigma_g)$ by first taking an infinite wedge of circles, one for each bordered $3$-manifold with boundary $\Sigma_g$, and then attaching a triangle for every surgery triangle $Y_{\alpha},Y_{\beta},Y_{\gamma}$.  Note that there is a crucial ambiguity here in orienting the attaching maps.  In section $4$ we prove the following result, from which Theorem \ref{minimality} follows as a corollary:

\begin{proposition} \label{homologyrank} There exists a choice of attaching maps so that $H_1(B(\Sigma_g);\Z)$ is a free abelian group of rank n(g). \end{proposition}

Interestingly, our proof of Proposition $\ref{homologyrank}$ does not rely on the existence of $M_g$.  Instead, it exploits a relationship between $X(\Sigma_g)$ and another geometry: the binary symplectic dual polar space, or $DSp(2g,2)$.  The points of this geometry are Lagrangian subspaces of $\F_2^{2g}$ and the lines are triples of distinct Lagrangians whose intersection has dimension $g-1$.  In section $3$ we construct an explicit surjective map $\mu: X(\Sigma_g) \to DSp(2g,2)$ which takes surgery triangles to lines (or possibly tripled points).  Under this map, any generating set for $X(\Sigma_g)$ maps to a generating set for $DSp(2g,2)$.  

It is a result of Brouwer that the ``universal embedding dimension'' of $DSp(2g,2)$ is at least $n(g)$.  The proof uses spectral graph theory and is sketched in \cite{BlBr}.  Theorem \ref{minimality} follows from Brouwer's result, and the fact that the universal embedding dimension is a lower bound for the cardinality of any generating set for $DSp(2g,2)$.  

In fact, it was shown by Li \cite{Li}, and independently Blokhuis and Brouwer \cite{BlBr}, that the universal embedding dimension of $DSp(2g,2)$ is exactly equal to $n(g)$.  Our Theorem \ref{minimalgen} implies a stronger result, that $DSp(2g,2)$ has a spanning set of cardinality $n(g)$.  This confirms a conjecture of Blokhuis and Brouwer.

It should be noted that Li (and earlier McClurg \cite{McC}) proposed a spanning set for $DSp(2g,2)$.  However, their set is different from the image of $M_g$, so we have not verified that it generates.  

The number $n(g)$ was initially suggested to us by some lengthy computer calculations.  The author would like to thank Gabriel Gaster for his help in carrying out those calculations.  

\section{Special Handlebodies}

In this section we define a special set of bordered $3$-manifolds.  We will construct these manifolds by doing surgery on links in a thickened punctured disk.   

Explicitly, let $D$ be the unit disk in $\R^2$, and let $D_g^*$ be the punctured disk obtained by removing $g$ smaller disks $p_1, p_2, \dots,p_g$ from the interior of $D$.  Thickening $D_g^*$ yields a handlebody $H_g = D_g^* \times [-1,1]$, whose boundary is a genus $g$ surface $\Sigma_g$.   We choose a collection of disjoint arcs $\alpha_1,\dots,\alpha_g \subset D_g^*$, which connect the punctures to the boundary of the disk as shown: 
\begin{center}\includegraphics[ scale = 0.5]{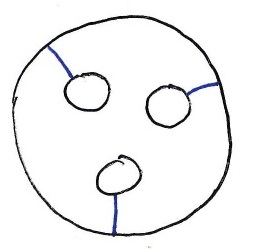}\end{center}
Let $A_i$ denote the disk $\alpha_i \times [-1,1]$, so that $A_1,\dots,A_g$ are a set of compressing disks for the handlebody $H_g$.  

\begin{definition} We say that a framed link $L \subset H_g$ is \emph{minimal} if it intersects each disk $A_i$ at most once geometrically, and each of its components intersects at least one disk.  We say that $L$ is \emph{crossingless} if it is contained in the slice $D_g^* \times \{0\}$.  \end{definition}

It can be shown that any manifold obtained by surgery on a minimal crossingless link is homeomorphic to a handlebody, and that if two minimal crossingless links are not isotopic then the corresponding handlebodies are not homeomorphic.  Strictly speaking we do not need either of these facts, but we will make use of them implicitly.

\begin{definition} A bordered manifold $H$ is said to be an \emph{almost-special handlebody} if it is obtained by doing $0$-surgery on each component of a minimal crossingless link $L \subset H_g$.    We denote the set of almost-special handlebodies by $S_g$. \end{definition}

The set of almost-special handlebodies was previously considered by Baldwin and Bloom \cite{BaBl}, although their description of these manifolds differs from ours.  We now describe a smaller set $M_g \subset S_g$.  The definition of $M_g$ requires several concepts, which we now explain.

Using the distinguished arcs $\alpha_i$, we can identify each puncture with a point on the boundary of $D$.  This identification induces a cyclic ordering on the set of punctures.  The definition of $M_g$ requires us to choose a total ordering that is compatible with this cyclic ordering.  Equivalently, we can choose a ``left-most'' puncture $p_1$, then list the remaining punctures $p_2, p_3, \dots, p_g$ as they appear in clockwise order around the boundary. 

Given a handlebody of genus $g$ and a puncture $p$, we can form a new handlebody of genus $g-1$ by filling $p$.  When we fill a puncture, the total ordering on the punctures of $D^*_g$ induces a total ordering on the punctures of $D^*_{g-1}$.  This allows us to induct on the genus of a handlebody.

\begin{definition} We say that an almost-special handlebody $H$ is \emph{reducible} if the corresponding crossingless link $L$ satisfies one of two conditions:
\begin{enumerate}
\item There is a puncture $p$ which is not circled by any component of $L$.
\item There is a component of $L$ which circles exactly one puncture.
\end{enumerate}
If $H$ is not reducible, we say that it is irreducible.  The \emph{reduction} of a reducible handlebody is the irreducible handlebody obtained by first deleting each component of $L$ that circles exactly one puncture, then filling every puncture that is not circled by any component.
\end{definition}

Given the ordering on punctures, we can order the components of any minimal crossingless link $L$, by declaring that a component $L_1$ is to the left of another component $L_2$ if the leftmost puncture circled by $L_1$ is to the left of the leftmost puncture circled by $L_2$.  It is therefore sensible to talk about the leftmost component of a minimal crossingless link. 

We now give an inductive definition of $M_g$.

\begin{definition} We say that an irreducible, almost-special handlebody is \emph{special} if the corresponding crossingless link $L$ satisfies each of the following three conditions:
\begin{enumerate} 
\item The handlebody obtained by deleting the left-most component of $L$ and filling every puncture it circles is special.
\item The punctures circled by the left-most component are consecutive with respect to the cyclic ordering (not necessarily with respect to the total ordering).
\item If the left-most component circles $p_1$, $p_g$, and $p_{g-1}$, then it circles every puncture.   
\end{enumerate}
In general, we say that an almost-special handlebody is special if its reduction is special.  We denote the set of special handlebodies by $M_g$. 
\end{definition} 

Note that the empty link in a genus zero handlebody satisfies each condition vacuously, so even though the definition is inductive it does not require a base case.  

\begin{proposition} For any $g \geq 1$, the number of irreducible special handlebodies is given by 
\[ m(g) = \frac{2^{g-1} +(-1)^g}{3}. \]
\end{proposition}
\begin{proof} It suffices to show that $m(g)$ satisfies the recursion
\[ m(g) = m(g-1) + 2m(g-2), \]
together with the initial conditions $m(1) = 0$, $m(2) = 1$.  The initial conditions are clear from the definitions.  To prove the recursion, observe that any irreducible element of $M_g$ is obtained in a unique way by one of three constructions:
\begin{enumerate}
\item Starting with an irreducible element of $M_{g-1}$, insert a puncture inside the left-most component, just to the right of the left-most puncture it contains.
\item Starting with an irreducible element of $M_{g-2}$, insert two new left-most punctures, and create a new component circling both of them.
\item Starting with an irreducible element of $M_{g-2}$, insert a new left-most puncture and a new right-most puncture, and create a new component circling both of them.
\end{enumerate}
\end{proof}

Be warned that the special case $m(0) = 1$ is not consistent with the formula above.  

\begin{proposition} The number of special handlebodies is given by the formula
\[ n(g) = \frac{(2^g+1)(2^{g-1}+1)}{3} \]
\end{proposition}
\begin{proof} A special handlebody is determined uniquely by a set uncircled punctures, a set of components that circle exactly one puncture, and an irreducible special handlebody.  Therefore, the number of special handlebodies is equal to:
\[ \sum_{k=0}^g \sum_{l=0}^{g-k} {g \choose k} {g - k \choose l} m(g-k-l) \] 
where $m(h)$ is the number of irreducible handlebodies of genus $h$.  This iterated summation can be carried out in an elementary (but tedious) manner, by repeatedly applying the binomial theorem.  The result is $n(g)$.
\end{proof}

\section{Generation}

In this section, we prove that $M_g$ generates $X(\Sigma_g)$.  First we introduce some algebraic notation for the manifolds we are considering.  

Observe that $X(\Sigma_g)$ has a simple binary operation, given by stacking thickened punctured disks.   Given bordered manifolds $Y_1$ and $Y_2$ we denote the result of stacking $Y_1$ on top of $Y_2$ by $Y_1Y_2$.   We denote the result of $0$-surgery on a minimal crossingless knot by the symbol $(i_1 i_2 \dots i_k)$, where $p_{i_1},p_{i_2},\dots,p_{i_k}$ are the punctures circled by the knot, listed in increasing order.  

A special handlebody can therefore be represented by a symbol like $(247)(56)(8)$.   On the other hand, not every symbol like this corresponds to a special handlebody.  For example, $(123)(13)$ is crossingless but not minimal, and $(13)(24)$ is minimal but not crossingless.  Note that stacking is not commutative, so for example $(13)(24) \neq (24)(13)$. 

Finally, we will also need to change the surgery coefficients on some links in our diagram.  To represent $p$-framed surgery on a minimal crossingless knot, we will use a symbol like $(247)_p$.

\begin{proposition} To show that $M_g$ generates $X(\Sigma_g)$, it suffices to show that the following two handlebodies lie in the span of $M_g$: 
\begin{itemize}
\item Handlebody $A$ = the genus $5$ handlebody $(145)(23)$.
\item Handlebody $B$ = the genus $6$ handlebody $(14)(23)(56)$.
\end{itemize} 
\end{proposition}
\begin{proof} By the result of Baldwin and Bloom, it suffices to show that every irreducible element of $S_g$ lies in the span of $M_g$.  Suppose that there is an irreducible $H$ in $S_g$ that is not generated by $M_g$.  Since $H$ is not special it violates one of the three conditions in the definition of special handlebody.  If it violates the third condition then we can simplify it using the reduction of handlebody $A$.  If it does not violate the third condition, but does violate the second condition, then we can simplify it using the reduction of handlebody $B$.  If it does not violate the second and third conditions, but does violate the first, then we can forget about the leftmost component and proceed by induction on $g$.  
\end{proof}

We are therefore reduced to doing two explicit computations.  There are two key results that make these computations possible.  

\begin{proposition} \label{genus3prop} The genus $3$ handlebody $(123)$ lies in the span of $(12)(23)$ and the reducible elements of $M_3$.  \end{proposition}
\begin{proof} For any set of punctures $X$ there is a surgery triangle relating $X_{p}$,$X_{p+1}$, and the empty diagram.  Therefore, $(123)$ lies in the span of the empty diagram and $(123)_1$, which in turn lies in the span of $(123)_{1}(13)$ and $(123)_{1}(13)_{-1}$.  The former is equivalent by a handleslide to $(1)_1(13)$, which lies in the span of the reducible diagrams $(1)(13)$ and $(13)$.  To simplify $(123)_1(13)_{-1}$ we can apply the famous lantern relation, in the form
\[ (123)_{1}(13)_{-1} = (12)_{1}(23)_{1}(1)_{-1}(2)_{-1}(3)_{-1} \]
Modifying surgery coefficients on the right hand side of this equation shows that it lies in the span of $(12)(23)$ and reducibles, as desired.  
\end{proof}

We can easily modify the argument above to show that $(123)$ lies in the span of $(23)(12)$ and reducibles.  Thus Proposition \ref{genus3prop} remains valid no matter how we permute the punctures.  

\begin{proposition} \label{genus4prop} The genus $4$ handlebody $(12)(34)$ lies in the span of $(13)(24)$, $(14)(23)$, $(1234)$, and the reducible elements of $M_4$. \end{proposition}
\begin{proof} By the result of Baldwin and Bloom, $(13)(24)$ lies in the span of the special handlebodies.  It is therefore enough to observe that there is a diffeomorphism $\phi$ of the standard genus $4$ handlebody which takes the diagram of $(12)(34)$ to the diagram of $(13)(24)$, and whose inverse sends the diagram of any special handlebody to a diagram which lies in the span of $(13)(24)$, $(14)(23)$, $(1234)$, and the reducible elements of $M_4$.

Such a diffeomorphism can be constructed as follows.  Arrange the punctures $1,2,3,4$ so that they lie on the vertices of a square, with $1$ in the upper right corner.  Draw a vertical arc that separates punctures $1$ and $4$ from punctures $2$ and $3$.  Flip over the part of $H_g$ which lies to the right of the arc, thereby switching punctures $2$ and $3$.  The result of applying this diffeomorphism to $(12)(34)$ is a two-component diagram whose components link punctures $1$ and $3$ and $2$ and $4$, respectively.  This diagram is not crossingless, however we can make it crossingless by flipping over punctures $2$ and $3$ individually, in a direction opposite to the original flip.  The combination of these three flips is a diffeomorphism $\phi: H_g \to H_g$ that takes $(12)(34)$ to $(13)(24)$.  

Finally, one applies the inverse diffeomorphism $\phi^{-1}$ to all special handlebodies of genus 4.  Using the result of Baldwin and Bloom for genera 2 and 3, one checks (tediously) that the resulting handlebodies all lie in the combined span of $(13)(24)$, $(14)(23)$, $(1234)$, and the reducible special handlebodies.
\end{proof}

We are now ready to eliminate handlebodies $A$ and $B$.

\begin{proposition}  \label{genus5prop} The handlebody $(145)(23)$ lies in the span of $M_5$. \end{proposition}
\begin{proof} We write $A \rightarrow X + Y + Z$ if the bordered manifold $A$ lies in the span of $X$,$Y$, $Z$, and manifolds already known to be in the span of $M_5$, like reducible handlebodies.  The reduction is very complicated to write down in full, so we only show the important steps.  The arrows below all follow from a combination of Propositions \ref{genus3prop} and \ref{genus4prop}:
\begin{eqnarray*}(145)(23) &\rightarrow& (15)(14)(23) \rightarrow (15)(12)(34) + (15)(13)(24) + (15)(1234) + (15)(234) \\
 (15)(13)(24) &\rightarrow& (13)(35)(24) \rightarrow (13)(23)(45) + (13)(25)(34) + (13)(2345) + (13)(245) \\
(13)(2345) &\rightarrow& (13)(23)(345) + (13)(23)(45) \\
(13)(23)(345) &\rightarrow& (12)(23)(345)  \rightarrow (12345) + (123)(45) + (12)(345) + (13)(245) \\
(12)(25)(34) &\rightarrow& (12)(35)(34) + (15)(23)(34) + (1235)(34) + (125)(34) \\
(1235)(34) &\rightarrow& (12345) + (125)(34) \end{eqnarray*} \end{proof}

\begin{proposition} \label{genus6prop} The handlebody $(14)(23)(56)$ lies in the span of $M_6$. \end{proposition}
\begin{proof} Again, we show the important parts of the reduction:
\begin{eqnarray*} (14)(23)(56) &\rightarrow& (13)(24)(56) + (12)(34)(56) + (1234)(56) \\
 (13)(24)(56) &\rightarrow& (13)(25)(46) + (13)(26)(45) + (13)(2456) \\
 (13)(25)(46) &\rightarrow& (12)(35)(46) + (15)(23)(46) + (1235)(46) \\
(13)(26)(45) &\rightarrow& (12)(36)(45) + (16)(23)(45) + (1236)(45) \\
 (13)(2456)&\rightarrow& (12)(3456) + (1456)(23) + (123456) \\
 (12)(35)(46) &\rightarrow& (12)(34)(56) + (12)(36)(45) + (12)(3456) \\
 (15)(23)(46) &=& (23)(15)(46) \rightarrow (23)(16)(45) + (23)(14)(56) + (23)(1456) \\
 (1235)(46) &\rightarrow& (1234)(56) + (1236)(45) + (123456) 
\end{eqnarray*}
The only byproduct of the above reductions which does not lie in $M_6$ is $(23)(1456)$.  Grouping together punctures $4$ and $5$ and applying Proposition \ref{genus5prop} shows that this handlebody lies in the span of $M_6$.  \end{proof}

\section{Minimality}

In this section we establish a lower bound for the cardinality of any generating set of $X(\Sigma_g)$.  Instead of working directly with $X(\Sigma_g)$ we use its ``universal embedding'' (see \cite{BlBr} or \cite{Li}, for example), or more precisely an integral lift of this embedding. 

\begin{definition} Let $\Sigma$ be a connected oriented surface.  We denote by $K(\Sigma)$ the free abelian group spanned by bordered $3$-manifolds with boundary $\Sigma$, modulo the relations
\[ Y_{\alpha} + Y_{\beta} + Y_{\gamma} = 0 \]
for every surgery triangle $(Y_{\alpha},Y_{\beta},Y_{\gamma})$. \end{definition}

There is a tautological map $X(\Sigma) \to K(\Sigma)$, and under this map any generating set for $X(\Sigma)$ maps to a spanning set for $K(\Sigma)$.  The group $K(\Sigma)$ can be thought of as a sort of Grothendieck group of bordered $3$-manifolds.

A subspace $L \subset H_1(\Sigma;\F_2)$ is said to be Lagrangian if it is isotropic with respect to the intersection pairing and has dimension $g$.  If $S \subset H_1(\Sigma;\F_2)$ is an isotropic subspace of dimension $g-1$, then the quotient $S^{\perp}/S$ is a $2$-dimensional $\F_2$ vector space, so $S$ is contained in exactly $3$ Lagrangian subspaces.  The relationship between these $3$ Lagrangians is analogous to the relationship between the $3$ manifolds involved in a surgery triangle, so we call such triples ``isotropic triangles''.

\begin{definition} Let $\Sigma$ be a connected oriented surface.  We denote by $L(\Sigma)$ the free abelian group spanned by Lagrangian subspaces $L \subset H_1(\Sigma;\F_2)$, modulo the relation
\[ L_{\alpha} + L_{\beta} + L_{\gamma} = 0 \]
for every isotropic triangle $(L_{\alpha}, L_{\beta}, L_{\gamma})$.  When $\Sigma = S^2$ we consider there to be a single Lagrangian subspace, the zero space, so that $L(S^2) = \Z$. \end{definition}

From now on all homology groups will have coefficients in $\F_2$, so $H_1(\Sigma) = H_1(\Sigma;\F_2)$. 

\begin{proposition} If $Y$ is a $3$-manifold with (possibly disconnected) boundary $\Sigma$, then the kernel $L(Y)$ of the map $H_1(\Sigma) \to H_1(Y)$ is Lagrangian. \end{proposition}
\begin{proof} Let $\alpha$ and $\beta$ be a curves on $\Sigma$, bounding surfaces $A$ and $B$ in $Y$.  If we perturb $A$ and $B$ to intersect transversely, then their intersection is a union of arcs, and these arcs together show that $\alpha \cap \beta$ has even cardinality.  This shows that $L(Y)$ is isotropic.   

Now consider the map 
\[ p: H_1(\Sigma)/L(Y) \to L(Y)^* \]
arising from the intersection pairing on $H_1(\Sigma)$.  To show that $L(Y)$ has dimension exactly $g$, and is therefore Lagrangian, it suffices to show that $p$ is injective.

Therefore, suppose that $\gamma \in H_1(\Sigma)$ is nonzero in $H_1(\Sigma)/L(Y)$.  Pushing forward by the inclusion $j:\Sigma \to Y$, we obtain a class $j_*(\gamma)$ in $H_1(Y)$.  This class is nonzero, because $\gamma$ does not lie in $L(Y) = \ker j_*$.   Therefore, by Poincar\'e duality, there is a class $A$ in $H_2(Y,\Sigma)$ such that $j_*(\gamma) \cdot A$ is nonzero.  Represent $A$ by a surface in $Y$ whose boundary lies on $\Sigma$, and let $\alpha = \partial A$.  Then $\alpha \in L(Y)$ and $\gamma \cdot \alpha = j_*(\gamma) \cdot A$ is nonzero, so $p(\gamma)$ is nonzero as desired.
\end{proof}

\begin{proposition} \label{triangleequalities} If $Y_{\alpha}$, $Y_{\beta}$, $Y_{\gamma}$ form a surgery triangle, then one of two possibilities holds:
\begin{enumerate}
\item The corresponding Lagrangians $L_{\alpha}$, $L_{\beta}$, $L_{\gamma}$ form an isotropic triangle and we have
\[ n_{\alpha} = n_{\beta} = n_{\gamma} \]
where $n_{\alpha}$ (for example) denotes the rank of $H_2(Y_{\alpha})$ as an $\F_2$ vector space.
\item The corresponding Lagrangians are equal and we have
\[ (-2)^{n_{\alpha}} + (-2)^{n_{\beta}} + (-2)^{n_{\gamma}} = 0 .\]
\end{enumerate}
\end{proposition}
\begin{proof} By assumption, the manifolds $Y_{\alpha}$, $Y_{\beta}$, $Y_{\gamma}$ are obtained by doing surgery on a knot $K$ in some bordered manifold $Y$ with boundary $\Sigma_g$.  Denote by $Z$ the manifold $Y \setminus \nu(K)$, whose boundary is $\Sigma_g \sqcup T^2$.  Let $\alpha$, $\beta$, and $\gamma$ be the surgery curves on $T^2$ whose fillings give rise to $Y_{\alpha}$, $Y_{\beta}$, and $Y_{\gamma}$.  

Let $\tilde{L}$ be the kernel of the inclusion map $H_1(T^2) \oplus H_1(\Sigma_g) \to H_1(Z)$.  Then $\tilde{L}$ is Lagrangian, so it has dimension $g+1$.  The kernel of the projection of $\tilde{L}$ onto $H_1(T^2)$ is an isotropic subspace $S \subset H_1(\Sigma_g)$, which is contained in all three Lagrangians $L_{\alpha}$,$L_{\beta}$,$L_{\gamma}$.  Since $S$ is obtained by intersecting $\tilde{L}$ with a codimension $2$ subspace, it has dimension at least $g-1$.  Since it is isotropic it has dimension at most $g$.  We therefore consider two separate cases.

If $S$ has dimension $g-1$, then the projection $\tilde{L} \to H_1(T^2)$ is surjective.  Hence there are classes $\alpha'$, $\beta'$, and $\gamma'$ in $H_1(\Sigma_g)$ that are homologous in $Z$ to $\alpha$, $\beta$, and $\gamma$.  Let $A$, $B$, and $C$ be surfaces witnessing these homologies.  Then we have (for example)
\[ \alpha' \cdot \beta' = A \cdot \beta' = A \cdot \beta = \alpha \cdot \beta = 1 ,\]
so the classes $\alpha'$, $\beta'$, and $\gamma'$ are all distinct.  Evidently these classes lie outside of $S$, and they bound in $Y_{\alpha}$, $Y_{\beta}$, and $Y_{\gamma}$ respectively, so together with $S$ they span the Lagrangians $L_{\alpha}$, $L_{\beta}$, and $L_{\gamma}$.  Since these Lagrangians are all distinct, they form an isotropic triangle.  Finally, note that a closed surface in $Y_{\alpha}$ (for example) cannot intersect the surgery curve in a homologically nontrivial way, because $\alpha$ does not bound in $Z$.  Therefore, the Mayer-Vietoris sequence
\[ H_2(T^2) \to H_2(Z) \oplus H_2( S^1 )  \to H_2(Y_{\alpha}) \to H_1(T^2) \]
shows that the map $H_2(Z)/H_2(T^2) \to H_2(Y_{\alpha})$ is an isomorphism, and similarly for $\beta$ and $\gamma$, so $n_{\alpha} = n_{\beta} = n_{\gamma}$.

If the dimension of $S$ is $g$, then the image of the projection $\tilde{L} \to H_1(T^2)$ is $1$-dimensional and $S = L_{\alpha} = L_{\beta} = L_{\gamma}$.   Without loss of generality, the image of the projection is spanned by $\alpha$.   Without loss of generality, $\alpha$ is in the image of the projection, so there is some $\alpha'$ in $H_1(\Sigma)$ such that $(\alpha, \alpha') \in \tilde{L}$.  By definition this means that $\alpha$ and $\alpha'$ are homologous in $Z$, so $\alpha'$ bounds in $Y_{\alpha}$, and therefore $\alpha' \in S = L_{\alpha}$.  We conclude that $\tilde{L} = \mathrm{span}(\alpha) \oplus S$.  

At any rate, we see that the kernel of $H_1(T^2) \to H_1(Z)$ is spanned by $\alpha$, hence the Mayer-Vietoris sequence
\[ H_2(Z) \oplus H_2(S^1) \to H_2(Y_{\alpha/\beta/\gamma}) \to H_1(T^2) \to H_1(Z) \oplus H_1(S^1) \]
shows that the rank of $H_2(Y_{\alpha})$ is one greater than the common rank of $H_2(Y_{\beta})$ and $H_2(Y_{\gamma})$.  The identity
\[ (-2)^{n_{\alpha}} + (-2)^{n_{\beta}} + (-2)^{n_{\gamma}} = 0 \]
then follows after a moment of thought about powers of $2$.
\end{proof}

Motivated by Proposition \ref{triangleequalities}, we define a map $\tilde{\mu}: X(\Sigma_g) \to L(\Sigma_g)$ by
\[ \tilde{\mu}(Y) = (-2)^{n(Y)} [L(Y)] \]
where $n(Y)$ is the rank of $H_2(Y)$ as a vector space over $\F_2$. 

\begin{proposition} The map $\tilde{\mu}$ induces a surjective homomorphism $\mu: K(\Sigma_g) \to L(\Sigma_g)$. \end{proposition}
\begin{proof} That $\tilde{\mu}$ induces a homomorphism follows directly from Proposition \ref{triangleequalities}.  That it is surjective follows from the fact that the homomorphism $\mathrm{Mod}_g \to Sp_{2g}(\F_2)$ is surjective, and the fact that $Sp_{2g}(\F_2)$ acts transitively on Lagrangian subspaces.
\end{proof}

In fact, $\mu$ is an isomorphism.  Before proving this we make a simple observation:

\begin{proposition} \label{keylemma} If $Y \in X(\Sigma_g)$ and $\alpha$ is any simple closed curve on $\Sigma_g$ whose homology class lies in $L(Y)$, then there exists $Y' \in X(\Sigma_g)$, equivalent to $Y$ in $K(\Sigma_g)$, such that $\alpha$ bounds an embedded disk in $Y'$. \end{proposition}
\begin{proof} Let $A \subset Y$ be a surface with boundary $\alpha$.  Without loss of generality, $A$ is nonorientable, hence it is homeomorphic to a connected sum of real projective planes.  Let $e_1,\dots,e_k$ denote the exceptional curves in these projective planes.  Removing a tubular neighborhood of each of these curves from $Y$, we obtain a manifold with $k$ torus boundary components.  With respect to the framings induced by $A$, the manifold $Y$ is obtained by filling curves of slope $(1,2)$ on each boundary torus.  Let $Y'$ be the result of filling the curves of slope $(1,0)$ instead.  Then the surgery triangles
\[Y_{(1,2)}(K) + Y_{(-1,-1)}(K) + Y_{(0,1)}(K) = 0  \]
\[Y_{(1,0)}(K) + Y_{(-1,-1)}(K) + Y_{(0,1)}(K) = 0, \]
which are valid for any framed knot $K$ in any $3$-manifold $Y$, together show that $Y' = Y$ in $K(\Sigma_g)$.  The effect of these surgeries on the surface $A$ is to blow down all exceptional curves, the result being an embedded disk $A' \subset Y'$ with boundary $\alpha$. 
\end{proof}

\begin{theorem} The map $\mu: K(\Sigma_g) \to L(\Sigma_g)$ is an isomorphism. \end{theorem}\begin{proof} First we treat the case $g=0$.  Since every $3$-manifold can be reduced to $S^3$ by surgery triangles, we know that $K(S^2)$ is a cyclic group.  Any generator of this cyclic group maps to a generator of $L(S^2) = \Z$, hence $K(S^2) = \Z$ and $\mu: K(S^2) \to L(S^2)$ is an isomorphism.

In general, any bordered manifold $Y_0$ with $L(Y_0) = L$ and $H_2(Y_0) = 0$ is equivalent in $K(\Sigma_g)$ to a ``standard'' example, namely the handlebody obtained by representing a standard basis of $L$ by disjoint simple closed curves on $\Sigma_g$ and attaching disks along these curves.  To see why, suppose that $\alpha_1,\dots,\alpha_g$ are curves representing a basis for $L$.  By Proposition \ref{keylemma}, $Y_0$ is equivalent in $K(\Sigma_g)$ to another bordered manifold $Y_1$ in which $\alpha_1$ bounds a disk.  Examining the proof of Proposition $\ref{keylemma}$, we see that $H_2(Y_0)$ and $H_2(Y_1)$ are isomorphic, so $H_2(Y_1) = 0$. 

Continuing this process inductively produces a manifold $Y_g$, equivalent to $Y_0$ in $K(\Sigma_g)$, such that all the curves $\alpha_i$ bound disks in $Y_g$.  This manifold $Y_g$ is the connected sum of a standard handlebody $H_g$ and a closed $3$-manifold $Z$ such that $H_2(Z) = 0$.  Because $\mu:K(S^2) \to L(S^2)$ is an isomorphism and $H_2(Z) = 0$, a connected sum with $Z$ is equivalent to a connected sum with $S^3$, hence $Y_g$ is equivalent to $H_g$ in $K(\Sigma_g)$.  Thus the original bordered manifold $Y$ is equivalent to $H_g$, as claimed.

Any isotropic triangle can also be realized by a standard example, due to the transitive action of $\mathrm{Mod}_g$ on isotropic triangles.  Hence there is an inverse map
\[ \mu^{-1}: L(\Sigma_g) \to K(\Sigma_g) \]
defined by sending a Lagrangian $L$ to any $Y$ with $L(Y) = L$ and $H_2(Y) = 0$.  Evidently we have $\mu^{-1}(\mu(Y)) = Y$ in $K(\Sigma_g)$, from which we conclude that $\mu$ is injective.
\end{proof}

Any generating set for $X(\Sigma_g)$ maps to a spanning set for $K(\Sigma_g)$, or equivalently for $L(\Sigma_g)$.  The group $L(\Sigma_g)$ has been computed by Blokhuis and Brouwer \cite{BlBr}, who found that it is a free abelian group of rank 
\[n(g) = \frac{(2^g+1)(2^{g-1}+1)}{3}.\] 
Since any spanning set for $X(\Sigma_g)$ maps to a spanning set for $K(\Sigma_g)$, we obtain: 

\begin{corollary} Any generating set for $X(\Sigma_g)$ has cardinality at least $n(g)$.  \end{corollary}


\begin{thebibliography}{9}
\bibitem{BaBl} Baldwin, J. and Bloom, J.  The monopole category and invariants of bordered 3-manifolds.  To appear.
\bibitem{BlBr} Blokhuis, A. and Brouwer, A.  The universal embedding dimension of the binary symplectic dual polar space. Journal of Discrete Mathematics, 2003, 246(1): 3-11.
\bibitem{Li} Li, P.  On the Universal Embedding of the $Sp_{2n}(2)$ Dual Polar Space.  Journal of Combinatorial Theory, 2001, 94 (1): 100-117.
\bibitem{Lic} Lickorish, W. B. R.  A representation of orientable combinatorial 3-manifolds.  Annals of Mathematics, 1962, 76 (3): 531-540.
\bibitem{McC} McClurg, P. On the Univeral Embedding of Dual Polar Spaces of Type $Sp_{2n}(2)$.  Journal of Combinatorics, 2000, 90 (1):104-122.
\end{thebibliography}
\end{document}